\documentclass{amsart}
\usepackage{amssymb}
\usepackage{amsmath}
\usepackage{latexsym}%
\usepackage[normalem]{ulem}
\usepackage{color}
\usepackage[all]{xy}
\usepackage[linktoc=all, pagebackref, hyperindex]{hyperref}%
\hypersetup{colorlinks,  citecolor=blue,  filecolor=blue,  linkcolor=red,  urlcolor=black}

\def\N{{\mathbb{N}}}
\def\Z{{\mathbb{Z}}}

\newcommand{\ff}{\mathfrak{f}}
\newcommand{\bu}{\mathbf{u}}
\newcommand{\bv}{\mathbf{v}}

\DeclareMathOperator{\Apr}{Apr}
\DeclareMathOperator{\PF}{PF}
\DeclareMathOperator{\F}{F}
\DeclareMathOperator{\FG}{FG}
\DeclareMathOperator{\G}{G}
\DeclareMathOperator{\sN}{N}

\newtheorem{thm}{Theorem}[section]
\newtheorem{cor}[thm]{Corollary}

\newtheorem{lem}[thm]{Lemma}
\newtheorem{prop}[thm]{Proposition}

\theoremstyle{definition}

\newtheorem{example}[thm]{Example}

\begin{document}
	
	
	\title{Coefficient Rings of Numerical Semigroup Algebras}
	
	\author{I-Chiau Huang}
	\address{Institute of Mathematics, Academia Sinica,
		6F, Astronomy-Mathematics Building,
		No. 1, Sec. 4, Roosevelt Road, Taipei 10617, Taiwan, R.O.C.}
	\email{ichuang@math.sinica.edu.tw}
	
	\author{Raheleh Jafari}
	\address{Mosaheb Institute of Mathematics, Kharazmi University, and 	
		School of Mathematics, Institute for
		Research in Fundamental Sciences (IPM), P. O. Box 19395-5746,
		Tehran, Iran}
	\email{rjafari@ipm.ir}
	
	\thanks{Raheleh Jafari was in part supported by a grant from IPM (No. 99130112).}
	
	\begin{abstract}  
		Numerical semigroup rings are investigated from the relative viewpoint.	
		 It is known that algebraic properties such as singularities of a numerical 
		semigroup ring are properties of a flat numerical semigroup algebra. In this 
		paper, we show that arithmetic and set-theoretic properties of a numerical 
		semigroup ring are properties of an equi-gcd numerical semigroup algebra.
		\end{abstract}

	
	\maketitle
	
	
\section{Introduction}


A {\em numerical semigroup} $S$ is a monoid generated by finitely many positive rational 
numbers. In the literature, it is often assumed that $S\subset\N$ and $\gcd(S)=1$. With 
 these assumptions, the numerical semigroup ring is defined as the subring
\begin{equation}\label{eq:nsr}
\kappa[\![\bu^S]\!]:=\Big\{\sum_{s\in S}a_s\bu^s \colon a_s\in\kappa\Big\}
\end{equation}
of the power series ring $\kappa[\![\bu]\!]$ over a field $\kappa$. 
In addition to $R=\kappa[\![\bu^S]\!]$, we also use the notation $S=\log_\bu R$ to indicate the correspondence between logarithms and exponents of semigroups and rings. 
Because of the correspondence, one may work either on the logarithmic side or the 
exponential side of the theory.  In this paper, we choose to stay in the exponential side 
in compliance with the previous study of ring theoretic properties such as
Cohen-Macaulayness, Gorensteiness and complete intersection \cite{HK}.
 The reader preferring semigroups can switch freely to the logarithmic side.
Our relative viewpoint on arithmetic and set-theoretic properties such as
symmetry, irreducibility and Cohen-Macaulay type suggests further developments
for numerical semigroups.

In the literature, the numerical semigroup ring (\ref{eq:nsr}) is commonly 
denoted as $\kappa[\![S]\!]$, or $\kappa[\![s_1,\ldots,s_n]\!]$ if $S$ is generated by 
$s_1,\ldots,s_n$. Although the common notation leads to no misunderstanding in the 
classical situation, the appearance of rings such as ${\mathbb Q}[\![3,5,7]\!]$ are 
vague in distinguishing elements $3,5,7$ in the underlying field $\mathbb Q$ from 
exponents $3,5,7$ of monomials $\bu^3,\bu^5,\bu^7$. To emphasize the relative 
nature, our notation is more handy in adding monomials to a numerical semigroup 
ring. For instance, given a numerical semigroup ring $R$ in the variable $\bu$, we may
add a monomial $\bu^t$ to $R$ and obtain $R[\![\bu^t]\!]$.

The assumptions $S\subset\N$ and $\gcd(S)=1$  for a numerical semigroup $S$ 
are not essential for the construction of $\kappa[\![\bu^S]\!]$. For an arbitrary numerical 
semigroup $S$, we may still define a numerical semigroup ring $\kappa[\![\bu^S]\!]$ as 
(\ref{eq:nsr}) allowing for monomials in $\bu$  to have rational exponents. Given a 
positive rational number $t$, the numerical semigroup ring $\kappa[\![\bv^{tS}]\!]$ is
isomorphic to $\kappa[\![\bu^S]\!]$ via the symbolized relation $\bv^t=\bu$. For a 
numerical semigroup $S$, there is a unique positive rational number $t$ such that 
$tS\subset\N$ and $\gcd(tS)=1$. Therefore any numerical semigroup ring
is isomorphic to one constructed from a numerical semigroup $S$ with the conditions 
$S\subset\N$ and $\gcd(S)=1$.

A numerical semigroup ring  has various algebra structures. It is  an algebra over the underlying field,  as well as over other subrings. We consider the category of numerical semigroup rings. A morphism in the category gives rise to an algebra.
 As seen in \cite{HJ, HK},  such a relative viewpoint sheds light to the algebraic properties of numerical semigroup rings. In this paper, we show that arithmetic and set-theoretic properties of a numerical semigroup 
ring are also relative in nature.  To be precise, we consider
 numerical semigroups $S\subset S'$. The numerical semigroup ring $R':=\kappa[\![\bu^{S'}]\!]$ 
is an algebra over the numerical semigroup ring $R:=\kappa[\![\bu^S]\!]$. We denote the algebra 
$R'$ together with the {\em coefficient ring} $R$ by $R'/R$ and call it a {\em numerical semigroup algebra}.
A numerical semigroup algebra and its coefficient ring can be represented using different variables. 
In this paper, we use the same variable for a numerical semigroup algebra and its coefficient ring.

In the classical case that $S\subset\N$ and $\gcd(S)=1$, there naturally arise two numerical semigroup algebras:
\begin{itemize}
\item
The ring $\kappa[\![\bu^S]\!]$ is an algebra over a Noether normalization 
$\kappa[\![\bu^s]\!]$, where $0\neq s\in S$.
\item
The ring $\kappa[\![\bu^S]\!]$ serves also as a coefficient ring for the algebra 
$\kappa[\![\bu]\!]$.
\end{itemize}
Singularities such as Cohen-Macaulayness, Gorensteiness and complete intersection 
of the ring $\kappa[\![\bu^S]\!]$ are in fact properties of the  flat algebra 
$\kappa[\![\bu^S]\!]/\kappa[\![\bu^s]\!]$. We may replace  the power series ring 
$\kappa[\![\bu^s]\!]$ by an arbitrary numerical semigroup ring $R$ and consider 
singularities of a numerical semigroup algebra $R''/R$ and its subalgebra $R'/R$. 
See \cite{HJ, HK} for investigations emphasizing algebras over a fixed coefficient ring.
The purpose of this paper is to clarify notions of the ring $\kappa[\![\bu^S]\!]$ that are 
in fact notions of the algebra $\kappa[\![\bu]\!]/\kappa[\![\bu^S]\!]$. We will replace  
the power series ring $\kappa[\![\bu]\!]$ by an arbitrary numerical semigroup ring 
$R''$ and regard it as an algebra over various coefficient rings. To emphasize the 
role of coefficient rings for the fixed $R''$, we call a coefficient ring $R'$ containing 
another coefficient ring $R$ an {\em extension} of $R$. For the case 
$R''=\kappa[\![\bu]\!]$, the semigroup $\log_\bu R'$ is called an oversemigroup in
\cite{oversemigroups} and also called an extension of $\log_\bu R$ in 
\cite{Ojeda-Rosales-2020}.

For a numerical semigroup $S\subset\N$ with $\gcd (S)=1$, there are many results 
concerning Frobenius numbers, pseudo-Frobenius numbers, gaps, symmetry, 
pseudo-symmetry, almost symmetry, irreducibility and Cohen-Macaulay type. On 
the way to exhibiting the relative nature of these invariants and notions, new notions 
appear. Let $R'/R$ be a numerical semigroup algebra in the variable $\bu$. 
Resembling $\kappa[\![\bu]\!]/\kappa[\![\bu^S]\!]$, we choose an integer $n$ such 
that $n\log_\bu R'\subset\N$. The algebra $R'/R$ is called {\em equi-gcd} if 
$\gcd(n\log_\bu R)=\gcd(n\log_\bu R')$. This condition is independent of the choice 
of $n$. In particular, $\kappa[\![\bu]\!]/\kappa[\![\bu^S]\!]$ is equi-gcd.  Arithmetic 
and set-theoretic properties of the semigroup $S$ often show up inside $\N$. We 
will prove that certain results in 
\cite{Froberg-etal-1987, Kunz-1970, Nari-2013, Ojeda-Rosales-2020, Rosales-1996, 
Rosales-Branco-2003, RG, Swanson-2009} 
for numerical semigroups are special cases of general phenomenons for coefficient 
rings of equi-gcd numerical semigroup algebras.  Along this line, generalizations of 
more results for numerical semigroups are expected.

Beyond the classical theory of numerical semigroups, there are new phenomenons 
in our framework. For example, the numerical semigroup ring 
$R=k[\![\bu^3,\bu^7,\bu^8]\!]$ is not symmetric as seen in Kunz's 
work \cite{Kunz-1970}. Adding the monomial $\bu^5$, we obtain a larger numerical 
semigroup ring $R'=R[\![\bu^5]\!]$, which is still not symmetric. 
We will see that the coefficient ring of the algebra $R'/R$ is symmetric. 
In this example, a symmetric phenomenon arises from two non-symmetric numerical
semigroup rings.

This paper is organized as follows:
For an equi-gcd numerical semigroup algebra, we will define Frobenius monomials, 
pseudo-Frobenius monomials, gap monomials and sporadic monomials in 
Section~\ref{sec:F&PF} as generalizations of corresponding invariants for numerical 
semigroups. In the classical case,  the number of gap  monomials has an upper 
bound in terms of the numbers of pseudo-Frobenius monomials and sporadic 
monomials \cite{Froberg-etal-1987}.  For a general equi-gcd numerical semigroup 
algebra, besides an analogous upper bound, we will  also give a lower bound for the 
number of gap monomials in terms of the numbers of Frobenius monomials and 
sporadic monomials. In this section, Kunz's criterion for a numerical semigroup ring 
to be Gorenstein \cite{Kunz-1970} is generalized.

 The next two sections are about symmetry, irreducibility and their  connection.
Section~\ref{sec:irr} is devoted to  the study of the irreducibility of coefficient rings of equi-gcd numerical semigroup algebras.   It is known  that irreducible numerical semigroups are either symmetric or pseudo-symmetric \cite{Rosales-Branco-2003}. As a main conclusion of this section, we generalize this result to 
equi-gcd numerical semigroup algebras.  Pseudo-Frobenius numbers of numerical semigroups can be 
read from Ap\'{e}ry sets. We present an analogous way to obtain pseudo-Frobenius monomials in Section~\ref{sec:sym},  which follows a characterization of almost symmetric semigroup algebras. 

In the last two sections, we will see classes of extensions. Radicals of  coefficient rings are introduced in Section~\ref{sec:rad} to generalize the classical result that a numerical semigroup is one over $n$ of 
infinitely many symmetric numerical semigroups \cite{Swanson-2009}. Finally, in Section~\ref{sec:fgm}, we investigate when all extensions of a coefficient ring are intersections of radicals. The classical version of this problem  was studied for numerical semigroups in \cite{Ojeda-Rosales-2020}. 
Fundamental gap monomials emerge as in the classical case.


\section{Frobenius and Pseudo-Frobenius Monomials}\label{sec:F&PF}


In this section, we consider a numerical semigroup algebra $R'/R$ in the variable $\bu$.
Recall that the conductor of $R'/R$ is the ideal
$\ff(R'/R):=\{a\in R \colon aR'\subset R\}$ of $R'$. Conductors are generated by
monomials and satisfy the following properties.
\begin{itemize}
\item $\ff(R'/R)=R'$ if and only if $R'=R$.
\item $\ff(R''/R')\ff(R'/R)\subset\ff(R''/R)\subset\ff(R'/R)$ for numerical semigroup algebras $R''/R'$ and $R'/R$.
\end{itemize}

\begin{example}
Let $R=\kappa[\![\bu^4,\bu^6,\bu^7,\bu^9]\!]$, $R'=\kappa[\![\bu^2,\bu^5]\!]$, $R''=\kappa[\![\bu^2,\bu^3]\!]$
and $R'''=\kappa[\![\bu]\!]$. Then
\begin{eqnarray*}
\ff(R'''/R')&=&R'''\bu^4=R'\bu^4+R'\bu^5,\\
\ff(R''/R')&=&R''\bu^2=R'\bu^2+R'\bu^5,\\
\ff(R'/R)&=&R'\bu^4+R'\bu^7=R\bu^4+R\bu^6+R\bu^7+R\bu^9.
\end{eqnarray*}
\end{example}

Monomials in $R'$ but not in $R$ are called {\em gap monomials} of $R'/R$. Monomials in
$R$ but not in $\ff(R'/R)$ are called {\em sporadic monomials} of $R'/R$. The set of gap
monomials and the set of sporadic monomials of $R'/R$ are denoted by $\G(R'/R)$ and
$\sN(R'/R)$, respectively.  For an equi-gcd numerical semigroup algebra
$\kappa[\![\bu]\!]/R$, sporadic monomials and gap monomials are exponential versions of
sporadic elements and gaps of the numerical semigroup $\log_\bu R$.  For a numerical
semigroup $S\subset\N$, the set of gaps is finite if and only if $\gcd(S)=1$. This well-known
fact is a special case of the following result for numerical semigroup algebras.
\begin{prop}\label{ff}
For a numerical semigroup algebra $R'/R$ in the variable $\bu$, the following conditions are equivalent.
\begin{itemize}
\item $R'/R$ is equi-gcd.
\item $\ff(R'/R)\neq 0$.
\item The set $\G(R'/R)$ is finite.
\item The set $\sN(R'/R)$ is finite.
\end{itemize}
\end{prop}
\begin{proof}
We may assume that  $\log_\bu R'\subset\N$ and $\gcd(\log_\bu R')=1$.  So
$\ff(\kappa[\![\bu]\!]/R')=\kappa[\![\bu]\!]\bu^{w'+1}$, where $w'$ is the Frobenius number
of $\log_\bu R'$. Let $p=\gcd(\log_\bu R)$.

Consider first the case that  $R'/R$ is equi-gcd, equivalently $p=1$. Let $w$ be the Frobenius number
of $\log_\bu R$. Then $\ff(\kappa[\![\bu]\!]/R)=\kappa[\![\bu]\!]\bu^{w+1}$ and
$\bu^{w+w'+1}\in\ff(R'/R)$. As a subset of
the finite set $\G(\kappa[\![\bu]\!]/R)$ (resp. $\sN(\kappa[\![\bu]\!]/R)$), the set 
$\G(R'/R)$ (resp. $\sN(R'/R)$) is finite.

Now we consider the case that  $R'/R$ is not equi-gcd, equivalently $p>1$. There are infinitely 
many monomials $\bu^s$ in $R'$ such that
$p$ does not divide $s$. Therefore $\G(R'/R)$ is infinite. Choose such a monomial $\bu^s$.
Then $\bu^t\bu^s\not\in R$ for any $\bu^t\in R$. That means $\bu^t\not\in\ff(R'/R)$. Therefore
$\ff(R'/R)=0$. The set $\sN(R'/R)$ consists of all monomials in $R$ and hence is also infinite.
\end{proof}

If an equi-gcd numerical semigroup algebra $R'/R$ is flat, then $R'=R$ by \cite[Proposition 3.5]{HJ}.
Beyond flatness, there are properties of non-trivial equi-gcd numerical semigroup algebras worth studying.
As seen in \cite{HJ, HK}, Ap\'{e}ry monomials form an essential tool to study flat numerical 
semigroup algebras. For equi-gcd numerical semigroup algebras, Frobenius monomials and 
Pseudo-Frobenius monomials play an important role.
In a numerical semigroup algebra $R'/R$ in the variable $\bu$, we have two partial orders:
\begin{itemize}
\item $\bu^s\preceq_R\bu^t$ if and only if $\bu^{t-s}\in R$.  	
\item $\bu^s\preceq_{R'}\bu^t$ if and only if $\bu^{t-s}\in R'$.
\end{itemize}
We consider maximal elements in $\G(R'/R)$ with respect to these two partial orders. Those maximal with
respect to $\preceq_{R'}$ are called {\em Frobenius monomials} of $R'/R$. Those maximal with
respect to $\preceq_{R}$ are called {\em pseudo-Frobenius monomials} of $R'/R$. Frobenius
monomials are pseudo-Frobenius monomials. The sets of Frobenius monomials and
pseudo-Frobenius monomials are denoted by $\F(R'/R)$ and $\PF(R'/R)$. The
cardinalities of these sets are called the {\em Frobenius type} (F-type for short) and the
{\em Cohen-Macaulay type} (CM-type for short) of $R'/R$, respectively.  
Cohen-Macaulay type was introduced for numerical semigroups  \cite{Froberg-etal-1987}.
This notion of type coincides with the Cohen-Macaulay type of a numerical semigroup ring \cite[Theorem 3.1]{Stamate-2018}.  The Frobenius type of an equi-gcd numerical semigroup
algebra $\kappa[\![\bu]\!]/\kappa[\![\bu^S]\!]$ is always one. Beyond the classical case, the
Frobenius type has an application somewhat dual to the Cohen-Macaulay type in the comparison
of the number of gap monomials and the number of sporadic monomials. See Proposition~\ref{prop:CM-F}.

Frobenius monomials and  pseudo-Frobenius monomials generalize Frobenius numbers and
pseudo-Frobenius numbers in the non-trivial case. Assume that $S\subset\N$ is a numerical 
semigroup with $\gcd(S)=1$. 
Recall that a number $w\in\Z\setminus S$ is 
called pseudo-Frobenius in \cite{RG} if $w+s\in S$ for any $0\neq s\in S$. For the trivial case 
$S=\N$, there is a unique Frobenius number  and a unique pseudo-Frobenius number, namely $-1$. 
However $\F(\kappa[\![\bu]\!]/\kappa[\![\bu]\!])=\PF(\kappa[\![\bu]\!]/\kappa[\![\bu]\!])=\emptyset$ 
in our definition and if $S\neq\N$, the only Frobenius monomial of 
$\kappa[\![\bu]\!]/\kappa[\![\bu^S]\!]$ is $\bu^{w}$, where $w$ is the Frobenius number of $S$. 
For the non-trivial case  $S\neq\N$, the exponents of pseudo-Frobenius 
monomials of $\kappa[\![\bu]\!]/\kappa[\![\bu^S]\!]$ are exactly pseudo-Frobenius numbers of $S$. 
We remark that all pseudo-Frobenius numbers are positive in the nontrivial case.
Note that $0\neq w+1\in S$ in such a case.  If there is a negative pseudo-Frobenius number $-s$, 
then  $0\neq w+1-s\in S$.  By induction, $w+1-ns\in S$ for all $n\geq1$. This is impossible, since $S\subset\N$.

\begin{example}\label{ex:97886}
Let $R=\kappa[\![\bu^4,\bu^6,\bu^9]\!]$ and $R'=R[\![\bu^5]\!]$.
Then $\G(R'/R)=\{\bu^5,\bu^{11}\}$,
$\sN(R'/R)=\{1,\bu^6\}$ and $\PF(R'/R)=\F(R'/R)=\{\bu^{11}\}$.
\end{example}

\begin{example}
Let $R=\kappa[\![\bu^5,\bu^7,\bu^9]\!]$ and $R'=R[\![\bu^8,\bu^{11}]\!]$. Then 
$\G(R'/R)=\{\bu^8,\bu^{11},\bu^{13}\}$, $\sN(R'/R)=\{1,\bu^5\}$
and  $\PF(R'/R)=\F(R'/R)=\{\bu^{11},\bu^{13}\}$.
\end{example}

\begin{example}\label{ex:2.5}
Let $R=\kappa[\![\bu^5,\bu^7,\bu^{11},\bu^{13}]\!]$ and  $R'=R[\![\bu^3]\!]$. Then 
$\sN(R'/R)=\{1,\bu^5\}$, $\G(R'/R)=\{\bu^3,\bu^6,\bu^8,\bu^9\}$, $\PF(R'/R)=\{\bu^6,\bu^8,\bu^9\}$ 
and $\F(R'/R)=\{\bu^8,\bu^9\}$.
\end{example}

 In the classical case, the ratio of $|\G(R'/R)|$ to $|\sN(R'/R)|$ is bounded above by the CM-type
of $R'/R$ \cite[Theorem 20]{Froberg-etal-1987}. For equi-gcd numerical semigroup algebras, this result still
holds. Moreover, the inverse of the F-type provides a lower bound for the ratio.
\begin{prop}\label{prop:CM-F}
Let $R'/R$ be a non-trivial equi-gcd numerical semigroup algebra of CM-type $n$ and F-type 
$n'$. Then
\[
\frac{1}{n'}\leq \frac{|\G(R'/R)|}{|\sN(R'/R)|}\leq n.
\]
\end{prop}
\begin{proof}
For any $\bu^s\in\G(R'/R)$, there exists a pseudo-Frobenius monomial $\bu^w$ such that $\bu^{w-s}\in R$.
Clearly $\bu^{w-s}\not\in\ff(R'/R)$.
The element $(\bu^{w-s},\bu^w)\in\sN(R'/R)\times\PF(R'/R)$ determines
$\bu^s$. Therefore
\[
|\G(R'/R)|\leq n|\sN(R'/R)|.
\]

Let $\bu^s\in\sN(R'/R)$. Since $\bu^s\not\in\ff(R'/R)$, there exists a monomial $\bu^v\in R'$ such that $\bu^{s+v}\notin R$, which  implies $\bu^v\notin R$. Setting $w=s+v$, we have $\bu^w\in\G(R'/R)$ and $\bu^{w-s}\in\G(R'/R)$.
Multiplying $\bu^w$ by a monomial of $R'$, we may assume that $\bu^w$ is a Frobenius monomial.
The element $(\bu^{w-s},\bu^w)\in\G(R'/R)\times\F(R'/R)$ determines
$\bu^s$. Therefore
\[
|\sN(R'/R)|\leq n'|\G(R'/R)|.
\]
\end{proof}

In the situation where $\kappa[\![\bu]\!]/R$ is equi-gcd, the numerical semigroup ring
$R$ is Gorenstein if and only if it has only one pseudo-Frobenius monomial. Gorensteiness
of $R$ has a criterion that $|\G(\kappa[\![\bu]\!]/R)|=|\sN(\kappa[\![\bu]\!]/R)|$.
We may generalize this classical criterion \cite{Kunz-1970} to the relative situation as follows.

\begin{thm}\label{thm:ume}
Let $R'/R$ be an equi-gcd numerical semigroup algebra with a single Frobenius monomial $\bu^w$.
The algebra has a single pseudo-Frobenius monomial if and only if  $|\G(R'/R)|=|\sN(R'/R)|$.
\end{thm}
\begin{proof}
Given  $\bu^s\in\sN(R'/R)$, there is $\bu^t\in\G(R'/R)$ such that $\bu^{t+s}\in\G(R'/R)$.
Since $\bu^w$ is the unique Frobenius monomial, $\bu^{w-t-s}\in R'$ and hence
$\bu^{w-s}=\bu^{w-t-s}\bu^t\in R'$. But $\bu^{w-s}\not\in R$, since $\bu^s\in R$ and
$\bu^w=\bu^{w-s}\bu^s\not\in R$. So we have an injective map $\sN(R'/R)\to\G(R'/R)$
sending $\bu^s$ to $\bu^{w-s}$.

If $|\G(R'/R)|=|\sN(R'/R)|$, the above map is surjective. The preimage of $\bu^t\in\G(R'/R)$
is $\bu^{w-t}$. Since $\bu^t\preceq_R\bu^w$, the Frobenius monomial is the only
pseudo-Frobenius monomial.

Now we assume that $\bu^w$ is the only pseudo-Frobenius monomial of $R'/R$.
Consider $\bu^s\in\G(R'/R)$. Then $\bu^{w-s}\in R$. But $\bu^w=\bu^{w-s}\bu^s\not\in R$.
Hence $\bu^{w-s}\in\sN(R'/R)$. This shows that the map $\sN(R'/R)\to\G(R'/R)$ above is
also surjective. Therefore $|\G(R'/R)|=|\sN(R'/R)|$.

\end{proof}

\begin{example}\label{ex:97887}
Let $R=\kappa[\![\bu^8,\bu^{12},\bu^{19},\bu^{21}]\!]$ and $R'=R[\![\bu^{22}]\!]$. Then
$\G(R'/R)=\{\bu^{22},\bu^{30},\bu^{34}\}$ and $\sN(R'/R)=\{1,\bu^8,\bu^{12}\}$. 
In this example, $|\G(R'/R)|=|\sN(R'/R)|$ even though $\PF(R'/R)=\F(R'/R)=\{\bu^{30},\bu^{34}\}$.
\end{example}

In examples~\ref{ex:97886} and \ref{ex:97887}, there is a monomial $\bu^s$ such that
$\G(R'/R)=\bu^s\sN(R'/R)$. This is not true in general for a numerical semigroup algebra
$R'/R$ such that $|\G(R'/R)|=|\sN(R'/R)|$.

\begin{example}
Let $R'=\kappa[\![\bu]\!]$ and $R=\kappa[\![\bu^3,\bu^5]\!]$. Then $\G(R'/R)=\{\bu,\bu^2,\bu^4,\bu^7\}$
and $\sN(R'/R)=\{1,\bu^3,\bu^5,\bu^6\}$. Note that $\G(R'/R)\neq \bu^s\sN(R'/R)$ for any $\bu^s\in R'$.
In this example, $\PF(R'/R)=\{\bu^7\}$.
\end{example}

\begin{example}\label{ex:22}
Let $R=\kappa[\![\bu^{11},\bu^{15},\bu^{19},\bu^{21}]\!]$
and  $R'=R[\![\bu^{31},\bu^{39}]\!]$. Then
$\G(R'/R)=\{\bu^{31},\bu^{39},\bu^{46},\bu^{50}\}$  and  $\sN(R'/R)=\{1,\bu^{11},\bu^{15},\bu^{19}\}$. Note that 
$\G(R'/R)\neq \bu^s\sN(R'/R)$ for any $\bu^s\in R'$.
 In this example, $\PF(R'/R)=\F(R'/R)=\{\bu^{46},\bu^{50}\}$.
\end{example}


\section{Irreducibility}\label{sec:irr}


 To study an algebraic structure, one often searchs for  simplest objects.  Irreducibility
is such a notion for numerical semigroups  and was first introduced in \cite{Rosales-Branco-2003}.
In this section, we extend the notion to the relative situation.
Let $R''/R$ be an equi-gcd numerical semigroup algebra.  For any extension $R'$ of $R$, the algebras 
$R''/R'$ and $R'/R$ are also equi-gcd. The coefficient ring $R$ is {\em irreducible} in $R''$ 
if it cannot be expressed as the intersection of two non-trivial extensions.  
For the case $R''=\kappa[\![\bu]\!]$, the coefficient ring $R$ is irreducible if and only if 
$\log_\bu R$ is an irreducible numerical semigroup. 
 Recall that, in the classical case $\kappa[\![\bu]\!]/R$, the algebra has a single Frobenius
monomial. 
\begin{example}
Let $R=\kappa[\![\bu^5,\bu^7,\bu^9]\!]$ and $R'=R[\![\bu^8]\!]$. Then $\G(R'/R)=\{\bu^8,\bu^{13}\}$.
Any  non-trivial extension of $R$ must contain
$\bu^{13}$. Therefore $R$ is irreducible as a coefficient ring of $R'$. Let $R''=R'[\![\bu^{11}]\!]$. Then
$\G(R''/R)=\{\bu^8,\bu^{11},\bu^{13}\}$ and $R[\![\bu^{11}]\!]\cap R[\![\bu^{13}]\!]=R$. 
Therefore $R$ is not irreducible as a coefficient ring of $R''$.
\end{example}

\begin{example}
Let $R=\kappa[\![\bu^8,\bu^{12},\bu^{19},\bu^{21}]\!]$. The algebra $R[\![\bu^{22}]\!]/R$ given in
Example~\ref{ex:97887} has gap monomials $\bu^{22},\bu^{30},\bu^{34}$. Since
$R=R[\![\bu^{30}]\!]\cap R[\![\bu^{34}]\!]$, it is not irreducible.  \end{example}

\begin{lem}\label{lem:single}
Let $R'/R$ be an equi-gcd numerical semigroup algebra. If the coefficient ring $R$ is irreducible,
then $R'/R$ has a single Frobenius monomial.
\end{lem}
\begin{proof}
A Frobenius monomial $\bu^w$ of $R'/R$ is maximal in $\G(R'/R)$ with respect to $\preceq_{R'}$,
hence $\bu^{2w}\in R$ and $\G(R[\![\bu^w]\!]/R)=\{\bu^w\}$. If $R'/R$ has two distinct
Frobenius monomials $\bu^{w_1}$ and $\bu^{w_2}$, then $R=R[\![\bu^{w_1}]\!]\cap R[\![\bu^{w_2}]\!]$ 
is not irreducible.
\end{proof}

The following criterion for irreducibility extends the classical case \cite{Rosales-Branco-2003}.

\begin{prop}\label{prop:max}
Let $\bu^w$ be a gap monomial of an equi-gcd numerical semigroup algebra $R'/R$. 
The following conditions are equivalent .
\begin{itemize}
\item
$R$ is irreducible with $\bu^w$ as the single Frobenius monomial.
\item
$R$ is maximal with respect to set inclusion
among coefficient rings with $\bu^w$ as the single Frobenius monomial. 
\item
$R$ is maximal with respect to set inclusion
among coefficient rings over which $\bu^w$ is a Frobenius monomial. 
\end{itemize}
\end{prop}

\begin{proof}
 For a Frobenius monomial $\bu^s\in\F(R'/R)$, note that $\G(R[\![\bu^s]\!]/R)=\{\bu^s\}$. 
In other words, joining a Frobenius monomial to a coefficient ring creates only one extra 
monomial. Therefore the set of coefficient rings with $\bu^w$ as the single Frobenius monomial and
the set of coefficient rings over which $\bu^w$ is a Frobenius monomial have the same maximal
coefficient rings with respect to set inclusion.

Let $R_1=R[\![\bu^w]\!]$ and $R_2$ be an extension of $R$ such that  $\bu^w\in\F(R'/R_2)$.
 If $\F(R'/R)=\{\bu^w\}$, then $\G(R_1/R)=\{\bu^w\}$ and $R=R_1\cap R_2$. If $R$ is  furthermore irreducible, then $R_2=R$. 
 So we obtain maximality. Let $R_3$ and $R_4$ be proper extensions of $R$. If $R$ is maximal among 
 coefficient rings of $R'$ with $\bu^w$ as  the single Frobenius monomial, then $\bu^w\in R_3\cap R_4$ and hence
$R\neq R_3\cap R_4$. Therefore maximality implies irreducibility.
\end{proof}

For an equi-gcd numerical semigroup algebra $R'/R$ with a single Frobenius monomial $\bu^w$,
we would like to find an irreducible coefficient ring $R_1$ containing $R$ such that $\bu^w$ is the single
Frobenius monomial of $R'/R_1$. If $R$ is not irreducible,  it is not maximal and there is  a monomial $\bu^s$ contained in 
an extension of $R$ such that  $\bu^{ns}\neq\bu^w$ for any $n\in\mathbb N$. In other words, $\bu^s$ 
is not a radical of $\bu^w$. As a step to construct an irreducible coefficient ring, the following result
generalizes \cite[Lemma~3.2]{Rosales-1996}.

\begin{lem}\label{lem:const}
Let $R'/R$ be an equi-gcd numerical semigroup algebra with a single Frobenius monomial
$\bu^w$.  Given a pseudo-Frobenius monomial $\bu^s$, 
 if it is not a radical of $\bu^w$, then
$\bu^w$ is also the single Frobenius monomial of $R'/R[\![\bu^s]\!]$.
\end{lem}
\begin{proof}
  For any $1\neq\bu^a\in R$, we have  $1\neq\bu^{s+a}\in R$ since $\bu^s\in\PF(R'/R)$. 
By induction, $\bu^{ns+a}\in R$ for $n\geq 1$.  Together with  the assumption, $\bu^w\not\in R[\![\bu^s]\!]$. 
 A gap monomial
$\bu^t$ of $R'/R[\![\bu^s]\!]$ is also a gap monomial of $R'/R$. Therefore $\bu^t\preceq_{R'}\bu^w$.
In other words, $\bu^w$ is the single Frobenius monomial of $R'/R[\![\bu^s]\!]$.
\end{proof}

\begin{thm}\label{thm:irr}
Let $R'/R$ be an equi-gcd numerical semigroup algebra in the variable $\bu$.  The coefficient ring
$R$ is irreducible if and only if  $\PF(R'/R)=\{\bu^w\}$ or $\PF(R'/R)=\{\bu^w,\bu^{w/2}\}$ for some
$\bu^w\in G(R'/R)$.  For CM-type one algebras over an irreducible coefficient ring, $|\G(R'/R)|=|\sN(R'/R)|$.
 For CM-type two algebras over an irreducible coefficient ring, $|\G(R'/R)|=|\sN(R'/R)|+1$.
\end{thm}
\begin{proof}
Consider the case that the coefficient ring $R$ is irreducible. By Lemma~\ref{lem:single},
$R'/R$ has a single Frobenius monomial $\bu^w$. It is shown in Proposition~\ref{prop:max}
that $R$ is maximal among numerical semigroup subrings of $R'$ in the variable $\bu$ with $\bu^w$
as  the single Frobenius monomial. Assume that there is a pseudo-Frobenius monomial $\bu^s$ other than $\bu^w$. By
Lemma~\ref{lem:const}, $w=ns$ for some $n\in\mathbb N$. Since $\bu^w$ is the unique Frobenius monomial,
$\bu^{ns-s}\in R'$. Since $\bu^{ns}\not\in R$ and $\bu^s$ is pseudo-Frobenius, the monomial $\bu^{ns-s}$
outside $R$ is also pseudo-Frobenius. By Lemma~\ref{lem:const} again, $n-1$ divides $n$. This implies that
$n=2$.

Consider the case that $\PF(R'/R)=\{\bu^w\}$. It is shown in Theorem~\ref{thm:ume} that
$|\G(R'/R)|=|\sN(R'/R)|$. Given $\bu^s\in\G(R'/R)$, we obtain $\bu^{w-s}\in R$. Therefore any subring of $R'$
properly containing $R$ has to contain $\bu^w$. By maximality, the coefficient ring $R$ is irreducible.

Consider the case that $\PF(R'/R)=\{\bu^w,\bu^{w/2}\}$.
Given $\bu^s\in\G(R'/R)$ not equal $\bu^{w/2}$, either $\bu^s\preceq_R\bu^w$ or
$\bu^s\preceq_R\bu^{w/2}$. If $\bu^s\preceq_R\bu^w$, then $\bu^{w-s}\in R$.
If $\bu^s\preceq_R\bu^{w/2}$, then $1\neq\bu^{(w/2)-s}\in R$ and hence
$\bu^{w-s}=\bu^{(w/2)-s}\bu^{w/2}\in R$. In either cases, $\bu^{w-s}\in R$.
But $\bu^{w-s}\not\in\ff(R'/R)$, since $\bu^w=\bu^s\bu^{w-s}\not\in R$. So we have an one-to-one map
$\G(R'/R)\setminus\{\bu^{w/2}\}\to\sN(R'/R)$ given by $\bu^s\mapsto\bu^{w-s}$.
As seen in the proof of Theorem~\ref{thm:ume}, $\bu^{w-t}\in\G(R'/R)$ for any $\bu^t\in N(R'/R)$.
Therefore the one-to-one map is also onto, whence the identity $|\G(R'/R)|=|\sN(R'/R)|+1$.
Given $\bu^s\in\G(R'/R)$ not equal $\bu^{w/2}$, we obtain $\bu^{w-s}\in R$.
As the CM-type one case, any subring of $R'$
properly containing $R$ has to contain $\bu^w$. By maximality, the coefficient ring $R$ is irreducible.
\end{proof}

\begin{example}
Let $R=\kappa[\![\bu^{n+1},\bu^{n+2},\ldots,\bu^{2n-1},\bu^{2n+1}]\!]$ and $R'=R[\![\bu^{n}]\!]$, 
 where $n>1$. If $n=2$, then $\G(R'/R)=\{\bu^2,\bu^4,\bu^7\}$ and $\PF(R'/R)=\{\bu^7\}$.
If $n>2$, then
 $\PF(R'/R)=\G(R'/R)=\{\bu^n,\bu^{2n}\}$. 
The coefficient ring $R$ is irreducible. 
\end{example}


\section{Symmetries}\label{sec:sym}


 Symmetry for a set can be often regarded as a phenomenon observed from one-to-one 
correspondence of elements in certain subsets determined by a simple rule.  Besides symmetric numerical semigroups that have been observed and well-studied, there are also pseudo-symmetry and almost symmetry properties standing out as variations. These 
notions of symmetry have a relative nature. Let $R'/R$ be an equi-gcd numerical semigroup algebra in the 
variable $\bu$. In view of Theorem~\ref{thm:ume}, symmetry of the coefficient ring $R$ is about
correspondence of elements in $\G(R'/R)$ and $\sN(R'/R)$. Pseudo-symmetry is about
correspondence of elements in $\G(R'/R)\setminus\{\bu^{w/2}\}$ and $\sN(R'/R)$, where
$\bu^w$ occurs as the only Frobenius monomial of $R'/R$. To be precise, we call the coefficient 
ring $R$ {\em symmetric}, if $\PF(R'/R)=\{\bu^w\}$. We  call $R$  {\em pseudo-symmetric},  if
$\PF(R'/R)=\{\bu^w,\bu^{w/2}\}$.  As a generalization of \cite{Rosales-Branco-2003}, Theorem~\ref{thm:irr}  asserts that a coefficient ring 
is irreducible if and only if it is symmetric or pseudo-symmetric. In this section, we consider almost 
symmetry for equi-gcd numerical semigroup algebras.

As a generalization of Ap\'{e}ry numbers to the relative case, {\em Ap\'{e}ry monomials} of a 
numerical semigroup algebra $R'/R$ can be described using the order $\preceq_R$ as minimal elements among monomials in $R'$. Among  Ap\'{e}ry monomials, maximal elements with 
respect to the order $\preceq_{R'}$ are called {\em maximal Ap\'{e}ry monomials}. It is known that a flat numerical semigroup 
algebra $R'/R$ in the variable $\bu$ is Gorenstein if and only if it has a unique maximal Ap\'{e}ry monomial 
\cite[Theorem 3.1]{HK}. Let $\kappa[\![\bu^S]\!]$ be a numerical semigroup ring and $s\in S$. It is also
known that $\kappa[\![\bu^S]\!]$ is a Gorenstein ring if and only if $\kappa[\![\bu^S]\!]/\kappa[\![\bu^s]\!]$
is a Gorenstein algebra. 

Assume that $S\subset\N$ and $\kappa[\![\bu]\!]/\kappa[\![\bu^S]\!]$ is equi-gcd. Kunz's symmetry 
criterion for the Gorenstein property of the ring $\kappa[\![\bu^S]\!]$ is in fact a statement of the algebra 
$\kappa[\![\bu]\!]/\kappa[\![\bu^S]\!]$. Let $\bu^w$ be the Frobenius monomial. Kunz shows that 
$\kappa[\![\bu^S]\!]$ is Gorenstein if and only if the map $S\to\Z\setminus S$ given by $s\mapsto w-s$
is bijective \cite{Kunz-1970}. Equivalently, $\kappa[\![\bu^S]\!]$ is Gorenstein if and only if the map 
$\sN(\kappa[\![\bu]\!]/\kappa[\![\bu^S]\!])\to\G(\kappa[\![\bu]\!]/\kappa[\![\bu^S]\!])$ given by 
$\bu^s\mapsto \bu^{w-s}$ is bijective. On the other hand, the algebra $\kappa[\![\bu]\!]/\kappa[\![\bu^S]\!]$ can describe
maximal Ap\'{e}ry monomials of $\kappa[\![\bu^S]\!]/\kappa[\![\bu^s]\!]$.

\begin{lem}\label{prop:PFApr}
Let $R'/R$ be an equi-gcd numerical semigroup algebra in the variable $\bu$. Given $1\neq\bu^s\in R$, 
\[
\PF(R'/R)=\{\bu^{t-s} : \bu^t\in\max_{\preceq_R}\left(\Apr(R/\kappa[\![\bu^s]\!])\setminus\Apr(R'/\kappa[\![\bu^s]\!])\right)\}.
\]
\end{lem}
\begin{proof}
The set $\bu^s\PF(R'/R)=\{\bu^{s+v} : \bu^v\in\PF(R'/R)\}$ consists of exactly those elements in $\bu^s\G(R'/R)=\{\bu^{s+v} : \bu^v\in\G(R'/R)\}$ maximal with 
respect to the order $\preceq_R$. Note that $\bu^s\PF(R'/R)\subset R$. Hence the set 
$R\cap\bu^s\G(R'/R)$ has same maximal elements with $\bu^s\G(R'/R)$. The lemma follows
from the fact that
$R\cap\bu^s\G(R'/R)=\Apr(R/\kappa[\![\bu^s]\!])\setminus\Apr(R'/\kappa[\![\bu^s]\!])$.
\end{proof}

For the case that $R'=\kappa[\![\bu]\!]$ in Lemma~\ref{prop:PFApr}, the set  $\Apr(R/\kappa[\![\bu^s]\!])\setminus\Apr(R'/\kappa[\![\bu^s]\!])$ 
consists of all elements $\bu^t\in\Apr(R/\kappa[\![\bu^s]\!])$ such that $t>s$. 
Let $\bu^w$ be the Frobenius monomial of $R'/R$.
Then $\bu^{w+s}\in\Apr(R/\kappa[\![\bu^s]\!])$. Given 
$\bu^t\in\Apr(R/\kappa[\![\bu^s]\!])$, if $t<s$, then $\bu^t\preceq_R\bu^{w+s}$.
So we recover \cite[Proposition~2.20]{RG} that
\[
\PF(\kappa[\![\bu]\!]/R)=\{\bu^{t-s} : \bu^t\in\max_{\preceq_R}\left(\Apr(R/\kappa[\![\bu^s]\!])\right)\}.
\]
In particular, $R/\kappa[\![\bu^s]\!]$ has a unique maximal Ap\'{e}ry monomial if and only if
$\kappa[\![\bu]\!]/R$ has a unique pseudo-Frobenius monomial.

 Note that maximal Ap\'{e}ry monomials of $R'/R$ can be characterized as those Ap\'{e}ry monomial 
$\bu^w$ such that $\bu^w\bu^{w'}$ is not Ap\'{e}ry for any non-trivial Ap\'{e}ry monomial $\bu^{w'}$. 
As an analogue, a pseudo-Frobenius monomial $\bu^w$ is called {\em maximal}, if $\bu^w\bu^{w'}$ is not 
pseudo-Frobenius for any pseudo-Frobenius monomial $\bu^{w'}$. Generalizing symmetry, 
 the coefficient ring $R$  is called {\em almost symmetric}  in $R'$, if the algebra has a unique maximal pseudo-Frobenius 
monomial.

When $R'=\kappa[\![\bu]\!]$, the coefficient ring $R$ is almost symmetric precisely when 
$\log_\bu R$ is  almost symmetric as it is defined in \cite{Barucci-Froberg-1997}.
Generalizing \cite[Theorem~2.4]{Nari-2013}, the next proposition justifies
our terminology.  Note that, over an almost symmetric coefficient ring, a numerical semigroup 
algebra has a  unique Frobenius monomial. 
By Theorem~\ref{thm:irr}, an irreducible coefficient ring of an equi-gcd numerical semigroup 
algebra is almost symmetric.  Therefore our notions extend the classical hierarchy that almost 
symmetric numerical semigroups 
include symmetric numerical semigroups and pseudo-symmetric numerical semigroups.

\begin{prop}\label{prop:almost-sym}
Let $R'/R$ be an equi-gcd numerical semigroup algebra with a single Frobenius monomial $\bu^w$. Write 
$\PF(R'/R)\setminus\{\bu^w\}=\{\bu^{t_1},\dots,\bu^{t_{n-1}}\}$, where $t_1<\cdots<t_{n-1}$. The following statements are equivalent.
\begin{itemize}
	\item $R$ is almost symmetric in $R'$.
		\item $\bu^{t_i}\bu^{t_{n-i}}=\bu^w$ for $i=1,\ldots,n-1$.
	\item $|\G(R'/R)|=|\sN(R'/R)|+n-1$.
\end{itemize}
\end{prop}

\begin{proof}
If $\bu^{t_i}\bu^{t_{n-i}}=\bu^w$ for all $i$, then all $\bu^{t_i}$ are not maximal. With the only maximal 
pseudo-Frobenius monomial $\bu^w$, the  coefficient ring $R$ is almost symmetric.

Conversely, if $R$ is almost 
symmetric, the only maximal pseudo-Frobenius monomial has to be $\bu^w$. For each $\bu^{t_i}$, 
there is an $\bu^{t_{j_1}}$ such that $\bu^{t_i}\bu^{t_{j_1}}$ is pseudo-Frobenius. If the product is not 
yet maximal, there is an $\bu^{t_{j_2}}$ such that $\bu^{t_i}\bu^{t_{j_1}}\bu^{t_{j_2}}$ is 
pseudo-Frobenius. Since $\bu^{t_{j_1}}\bu^{t_{j_2}}\not\in R$ and $\bu^s\bu^{t_{j_1}}\bu^{t_{j_2}}\in R$ 
for any $1\neq \bu^s\in R$, we obtain a pseudo-Frobenius monomial $\bu^{t_{j_1}+t_{j_2}}$. Keep
multiplying pseudo-Frobenius monomials, we reach some $\bu^{t_j}$ such that $\bu^{t_i}\bu^{t_j}=\bu^w$. 
Sorting the numbers $w-t_i$, we see $j=n-i$.

There is a map 
\[
\varphi:\sN(R'/R)\cup(\PF(R'/R)\setminus\{\bu^w\})\longrightarrow\G(R'/R)
\]
sending $\bu^s$ to $\bu^{w-s}$. Since the map is injective, 
\begin{equation}\label{eq:98515}
|\G(R'/R)|\geq |\sN(R'/R)|+n-1 
\end{equation}
and the equality holds if and only if $\varphi$ is surjective. The image of $\sN(R'/R)\setminus\{1\}$ 
is contained in $\G(R'/R)\setminus\PF(R'/R)$. If $\varphi$ is surjective, every $\bu^{t_i}$ is the 
image of some $\bu^{t_j}$. Taking the increasing sequence into account, this implies that $t_i+t_{n-i}=w$. 
Conversely, assume that $t_i+t_{n-i}=w$. Then $\PF(R'/R)$ is in the image of $\varphi$. To show  that
$\varphi$ is surjective, we take an element $\bu^v\in\G(R'/R)\setminus\PF(R'/R)$. Then 
$\bu^s\bu^v\in\PF(R'/R)$ for some $1\neq\bu^s\in R$. From the assumption, either $\bu^{s+v}=\bu^w$ 
or $\bu^{w-s-v}\in\PF(R'/R)$. Since $\bu^{w-v}=\bu^s\bu^{w-s-v}\in R$ and 
$\bu^{w-v}\bu^v=\bu^w\not\in R$, we obtain an element $\bu^{w-v}\in\sN(R'/R)$ such that
$\varphi(\bu^{w-v})=\bu^v$.
\end{proof}	
For a numerical semigroup algebra $R'/R$ with F-type one, $|\PF(R'/R)\setminus\F(R'/R)|$ provides 
a lower bound for $|\G(R'/R)|-|\sN(R'/R)|$.  The inequality does not hold in general. The number $|\G(R'/R)|-|\sN(R'/R)|$ 
is not even necessarily positive, when the F-type is greater than one.

\begin{example}
Let $R'=\kappa[\![\bu^{10},\bu^{11},\bu^{12},\bu^{13},\bu^{14},\bu^{15}]\!]$ and $R$ be its coefficient 
ring such that $\G(R'/R)=\{\bu^{12},\bu^{13},\bu^{23},\bu^{27}\}$. Then
$\sN(R'/R)=\{1,\bu^{10},\bu^{11},\bu^{14},\bu^{15}\}$  has more elements than $\G(R'/R)$. Here the F-type is two.
\end{example}

\begin{example}
For $n\in\mathbb N$, let $R=\kappa[\![\bu^{2n+1},\bu^{2n+2},\ldots,\bu^{4n+1}]\!]$ and
$R'=R[\![\bu^2]\!]$. Then $\PF(R'/R)=\G(R'/R)=\{\bu^2,\bu^4,\ldots,\bu^{2n}\}$ and
$\F(R'/R)=\{\bu^{2n}\}$. The  coefficient ring $R$ is almost symmetric. Note that  $R$ is irreducible if and only if $n\leq 2$.
\end{example}


\section{Radicals}\label{sec:rad}


Let $R''/R$ be a numerical semigroup algebra in the variable $\bu$. To illustrate the relative 
nature of an extension $R'$ of $R$, we often write $R'=R[\![\bu^{w_1},\ldots,\bu^{w_n}]\!]$ by 
joining $\bu^{w_1},\ldots,\bu^{w_n}\in\G(R''/R)$ to $R$. Clearly, $\G(R''/R)$ is the disjoint union 
of $\G(R''/R')$ and $\G(R'/R)$. If $\bu^{w_1},\ldots,\bu^{w_n}\in\F(R''/R)$, then 
$\G(R'/R)=\F(R'/R)=\{\bu^{w_1},\ldots,\bu^{w_n}\}$. Recall that a monomial $\bu^s\in R''$ is the 
$n$th radical of a monomial in $R$ if $\bu^{ns}\in R$. Besides trivial  extensions $R$ and $R''$, 
radicals of the coefficient ring $R$ provide more interesting examples of extensions. Given 
$n\in\N$, the extension of $R$ in $R''$ generated by all $n$th radicals of monomials of $R$ is 
called the $n$th {\em radical} of $R$ in $R''$. If $R''$ is understood from the context, we denote 
the $n$th radical by $\sqrt[n]{R}$.

For the case $R''=\kappa[\![\bu]\!]$,  radicals are investigated in the logarithmic form as quotients 
of the numerical semigroup  $\log_\bu R$. It is known that a numerical semigroup is one half of a 
symmetric numerical semigroup \cite{Rosales-GarciaSanchez-2008}. In fact, there are infinitely 
many such symmetric numerical semigroups \cite{Rosales-GarciaSanchez-2008-inf}. Even more, 
for each $n\geq 2$, a numerical semigroup can be expressed as quotients of infinitely many 
symmetric numerical semigroups by $n$ \cite[Theorem~5]{Swanson-2009}. For each $n\geq 3$, 
a numerical semigroup can be also expressed as quotients of infinitely many pseudo-symmetric 
numerical semigroups by $n$ \cite[Theorem~6]{Swanson-2009}. In this section, we show that  
the classical results in \cite{Swanson-2009} have  a relative nature.

\begin{example}
Let $R=\kappa[\![\bu^6,\bu^{10},\bu^{14},\bu^{2n+1}]\!]$, where $n\in\N$. Then 
$\sqrt[2]{R}=\kappa[\![\bu^3,\bu^5,\bu^7]\!]$ in $\kappa[\![\bu^3,\bu^4,\bu^5]\!]$. 
\end{example}

Numerical duplications are a source of radicals of coefficient rings. They are first introduced in 
the logarithmic form~\cite{Duplication}. In our notation, a  {\em numerical duplication} of a 
numerical semigroup ring $R$ is the numerical semigroup ring 
$R'=R[\![\bu^{s_1}\bu^{s_0/2},\ldots,\bu^{s_n}\bu^{s_0/2}]\!]$, where 
$\bu^{s_0},\ldots,\bu^{s_n}\in R$ but $\bu^{s_0/2}\not\in R$. In $R'$, we have $\sqrt[2]{R}=R'$.
Flatness of numerical semigroup algebras   is a main theme for the study of singularities.
However, as seen in the following proposition, there are no extensions of the coefficient ring
other than radicals in a flat algebra.

\begin{prop}
Let $R''/R$ be a numerical semigroup algebra and let $R'$ be an extension of the coefficient ring $R$.
If $R''/R'$ and $R''/R$ are both flat, then $R'$ is a radical of $R$.
\end{prop}
\begin{proof}
We may assume that the numerical semigroup rings are in the variable $\bu$ and $\log_\bu R''\subset\N$.
Let $d''=\gcd(\log_\bu R'')$, $d'=\gcd(\log_\bu R')$ and $d=\gcd(\log_\bu R)$. 
An element of $R'$ can be written
as $\bu^{sd'}$ for some $s\in\N$. Let $q=d/d'$. We claim $R'=\sqrt[q]{R}$. On the one hand, 
$\bu^{sd'q}=\bu^{sd}\in\kappa[\![\bu^{d/d''}]\!]\cap R''=R$,
where the last equality holds by flatness of $R''/R$  \cite[Proposition 3.5]{HJ}.
Therefore $R'\subset\sqrt[q]{R}$. On the other hand, consider $\bu^{sd'}\in\sqrt[q]{R}$. Then $\bu^{sd}\in R$
and hence $s\in\N$. The inclusion $\sqrt[q]{R}\subset R'$ follows from $\bu^{sd'}\in\kappa[\![\bu^{d'/d''}]\!]\cap R''=R'$,
where the last equality holds by flatness of $R''/R'$ [{\em ibid.}].
\end{proof}

Note that an equi-gcd algebra $\kappa[\![\bu]\!]/\kappa[\![\bu^S]\!]$ is never flat, unless $S=\N$. 
We can say more about radicals for equi-gcd numerical semigroup algebras in general.
Let $n\geq 2$ and $R''/R'$ be an equi-gcd numerical semigroup algebra in the variable $\bu$. 
Using a similar idea as \cite{Swanson-2009}, we will show that there are infinitely many equi-gcd 
numerical semigroup algebras $R''/R$ such that $R'=\sqrt[n]{R}$ and $R$ is symmetric in $R''$. 
We may assume that  $\log_\bu R''\subset\N$ and $\gcd(\log_\bu R'')=1$. In its logarithmic form, 
the condition $R'=\sqrt[n]{R}$ can be stated as $nS'=\log_\bu R\cap n\log_\bu R''$, where 
$S'=\log_\bu R'$. Since $R''/R$  is required to be equi-gcd, there  should be a monomial 
$\bu^h\in R''$ such that $\bu^s\in R$ for $s>h$.  To construct a symmetric coefficient ring $R$, 
we will choose an $h$ and  join to 
\[
R_0:=\kappa[\![\bu^{nS'},\bu^{h+1},\bu^{h+2},\dots,\bu^{2h+1}]\!]
\]
certain monomials $\bu^s\in R''$ satisfying $s\not\in n\log_\bu R''$. Note that $R''/R_0$ is 
equi-gcd. Denote $\{\bu^{w}\}=\F(\kappa[\![\bu]\!]/R')$. In other words, $w$ is the Frobenius 
number of $\log_\bu R'$. If  $h>w$, then $R_0\subset R'$.  If $h>nw$, then $nS'=\log_\bu R_0\cap n\log_\bu R''$.  
The algebra $\kappa[\![\bu]\!]/R'$ has finitely many gap monomials $\bu^{w_1},\bu^{w_2},\ldots$. 
 If $h>nw+w$, then
\[
R_1:=R_0[\![\bu^{h-nw_1},\bu^{h-nw_2},\ldots]\!]\subset R'.
\]
\begin{lem}
Assume that $h>2nw$. Then 
$\G(R_1/R_0)=\{\bu^{h-nw_1},\bu^{h-nw_2},\ldots\}$.
\end{lem}
\begin{proof}
The condition $h>2nw$ implies $2h-nw_i-nw_j>h$. Hence $\bu^{h-nw_i}\bu^{h-nw_j}\in R_0$.
 Monomials of $R_0$ are of the form $\bu^t$ or $\bu^{ns}$, where $t>h$ and $\bu^s\in R'$. 
 Clearly, $\bu^t\bu^{h-nw_i}\in R_0$. If $s>w_i$, then $ns+h-nw_i>h$ and hence $\bu^{ns}\bu^{h-nw_i}\in R_0$. If $s<w_i$, 
then $\bu^{w_i-s}\in\G(\kappa[\![\bu]\!]/R')$ and $\bu^{ns}\bu^{h-nw_i}=\bu^{h-nw_j}$ for some $j$. We 
conclude that joining $\bu^{h-nw_1},\ldots,\bu^{h-nw_m}$ to $R_0$ does not create other gap 
monomials.
\end{proof}
Now we assume furthermore $n\nmid h$ besides the condition $h>2nw$. Then $\bu^h\in\G(R'/R_1)$
and $nS'=\log_\bu R_1\cap n\log_\bu R''$.   Before joining to $R_1$ more monomials, we consider a monomial 
$\bu^s\in\G(R''/R')\subset\G(\kappa[\![\bu]\!]/R')$.  By the assumption $h>2nw$, we have
$h-s\geq h-w>w$ and hence $\bu^s\preceq_{R''}\bu^h$. This implies $\F(R''/R_1)\subset R'$.
Let $\bu^{t_1},\bu^{t_2},\ldots$ be those monomials in $\G(R'/R_1)$ satisfying 
$\bu^{t_i}\npreceq_{R''}\bu^h$. Define
\[
R_2:=R_1[\![\bu^{t_1},\bu^{t_2},\ldots]\!]\subset R'.
\]
Then $\G(R_2/R_1)=\{\bu^{t_1},\bu^{t_2},\ldots\}$. 
By the choice of $t_i$, the algebra $R''/R_2$ has a single Frobenius monomial $\bu^h$.

\begin{thm}
 Let $R''/R'$ be an equi-gcd numerical semigroup algebra. Given $n\geq 2$, the coefficient ring 
$R'$ is the $n$th radical of infinitely many coefficient ring $R$, which can be taken to be symmetric in $R''$. 
If $n\geq 3$, $R$ can be taken to be pseudo-symmetric. 
\end{thm}
\begin{proof}
For $h\in\N$ satisfying $n\nmid h$ and  $h>2nw$, we construct  $R_1$ and $R_2$ as above so that 
$\F(R''/R_2)=\{\bu^h\}$. By Proposition~\ref{prop:max}, there is an irreducible coefficient ring 
$R$ of $R''$ extending $R_2$ 
 such that $\F(R''/R)=\{\bu^h\}$. 
Clearly, $\sqrt[n]{R}\supset\sqrt[n]{R_1}=R'$. 
On the other hand, if $\bu^{ns}\in R$
for some $\bu^s\in R''$,
then $\bu^s\in R'$; otherwise $\bu^{h-ns}\in R_1\subset R_2\subset R$ contradicting to $\bu^h\not\in R$.
Therefore $\sqrt[n]{R}=R'$.  In particular,  $R\subset R'$. Since $R$ is irreducible, 
 $\PF(R'/R)\subset\{\bu^h,\bu^{h/2}\}$ by Theorem~\ref{thm:irr}. If we choose $h$ such that 
$\bu^{h/2}\notin R''$, then $\PF(R''/R)=\{\bu^h\}$, which means $R$ is symmetric in $R''$. 
There are infinitely many such $h$. For instance, we may choose $h$ to be an odd integer.
If $\bu^{h/2}\in R''$, then $\bu^{h/2}\in \G(R''/R)$ and $\bu^{h/2}\npreceq_{R}\bu^h$. Therefore,  
$\PF(R'/R)=\{\bu^h,\bu^{h/2}\}$,  i.e. $R$ is pseudo-symmetric in $R''$.   
If  $n\geq 3$, there are infinitely many such $h$. For instance, we may choose
$h$ to be of the form $ns+2$, where $s\in\N$.
\end{proof}

\begin{example}
Consider $R'=\kappa[\![\bu^4,\bu^6,\bu^7,\bu^9]\!]$ and $R''=\kappa[\![\bu^4,\bu^5,\bu^6,\bu^7]\!]$. 
Then $\F(\kappa[\![\bu]\!]/R')=\{\bu^5\}$. In the above construction for $n=2$ and $h=21$, we have
\begin{eqnarray*}
R_0&=&\kappa[\![\bu^8,\bu^{12},\bu^{14},\bu^{18},\bu^{23},\bu^{25},\bu^{27},\bu^{29}]\!], \\
R_1&=&R_0[\![\bu^{11},\bu^{15},\bu^{17},\bu^{19}]\!]=
\kappa[\![\bu^8,\bu^{11},\bu^{12},\bu^{14},\bu^{15},\bu^{17},\bu^{18}]\!]
\end{eqnarray*} 
and $R_2=R_1$. Then $R=R_2$ is symmetric in $R''$ and $\sqrt[2]{R}=R'$.
For $n=3$ and $h=31$ we have 
\begin{eqnarray*}
	R_0&=&\kappa[\![\bu^{12},\bu^{18},\bu^{21},\bu^{27},\bu^{32},\bu^{34},\bu^{35},\bu^{37},\bu^{38},\bu^{40},\bu^{41},\bu^{43}]\!], \\
	R_1&=&R_0[\![\bu^{16},\bu^{22},\bu^{25},\bu^{28}]\!]=
\kappa[\![\bu^{12},\bu^{16},\bu^{18},\bu^{21},\bu^{22},\bu^{25},\bu^{27},\bu^{35}]\!],\\
R_2&=&R_1[\![\bu^{29}]\!]=\kappa[\![\bu^{12},\bu^{16},\bu^{18},\bu^{21},\bu^{22},\bu^{25},\bu^{27},\bu^{29},\bu^{35}]\!].
\end{eqnarray*} 
Note that $R_2$ is not symmetric in $R''$.  Indeed,
$\PF(R''/R_2)=\{\bu^{17}, \bu^{20}, \bu^{23},\bu^{26},\bu^{31}\}$.
However, $R=R_2[\![\bu^{17}, \bu^{20}, \bu^{23},\bu^{26}]\!]$ is symmetric in $R''$ and $\sqrt[3]{R}=R'$.

\end{example}


\section{Fundamental Gap Monomials}\label{sec:fgm}


Let $R''/R$ be a numerical semigroup algebra in the variable $\bu$. Intersections of radicals of $R$ have been studied in 
\cite{Ojeda-Rosales-2020} for the case $R''=\kappa[\![\bu]\!]$ under the 
name ``arithmetic extension'' from the logarithmic viewpoint. In this section, we exhibit the relative nature
of the connection between arithmetic extensions and fundamental gaps.

\begin{example}
Let $R=\kappa[\![\bu^{7},\bu^{8},\bu^{9},\bu^{10},\bu^{11},\bu^{13}]\!]$ and $R''=R[\![\bu^3,\bu^{4}]\!]$.
Then $\G(R''/R)=\{\bu^3,\bu^4,\bu^6,\bu^{12}\}$. The  extension
$R[\![\bu^{12}]\!]=R[\![\bu^4]\!]\cap R[\![\bu^3]\!]$  of $R$ is the intersection $\sqrt[2]{R}\cap\sqrt[3]{R}$ of radicals.
\end{example}

\begin{example}
Let $R=\kappa[\![\bu^5,\bu^6,\bu^{13}]\!]$ and $R''=R[\![\bu^7,\bu^8]\!]$. Then 
$\G(R''/R)=\{\bu^7,\bu^8,\bu^{14}\}$. The  coefficient ring $R$ has three non-trivial extensions
 in $R''$. The extensions $R[\![\bu^7]\!]$ and 
$R[\![\bu^{14}]\!]$ are not intersection of radicals, while $R[\![\bu^8]\!]=\sqrt[2]{R}$.
\end{example}

\begin{example}
The coefficient ring $\kappa[\![\bu^5,\bu^6]\!]$ of $\kappa[\![\bu^5,\bu^6,\bu^{19}]\!]$ has only
trivial extensions.
The extension $\kappa[\![\bu^5,\bu^6,\bu^{19}]\!]$ of $\kappa[\![\bu^5,\bu^6]\!]$ in $\kappa[\![\bu]\!]$ is not an intersection of radicals.
\end{example}

\begin{example}
The coefficient ring $\kappa[\![\bu^5,\bu^6,\bu^{14}]\!]$ of $\kappa[\![\bu^5,\bu^6,\bu^7,\bu^8]\!]$ has two
non-trivial extensions $\kappa[\![\bu^{5},\bu^6,\bu^7]\!]$ and $\kappa[\![\bu^5,\bu^6,\bu^8]\!]$.
 Both are not intersection of radicals.
\end{example}

Let $R'/R$ be an equi-gcd numerical semigroup algebra.   The elements in
\[
\FG(R'/R):=\G(R'/R)\cap\sqrt[2]{R}\cap\sqrt[3]{R}
\] 
are called the {\em fundamental gap monomials} of $R'/R$. The gap monomial with the largest exponent 
is a fundamental gap monomial. 
For the case $R'=\kappa[\![\bu]\!]$, the exponents of monomials in $\FG(R'/R)$ 
are the fundamental gaps of the numerical semigroup  $\log_\bu R$ first
introduced in \cite{Fundamental-2004}.  Note that  $\FG(R'/R)\subset\sqrt[n]{R}$
for $n\geq 2$. Therefore any intersection of radicals of $R$ contains $\FG(R'/R)$.

\begin{example}
Let $R=\kappa[\![\bu^5,\bu^6]\!]$ and $R'=R[\![\bu^7,\bu^8]\!]$. Then 
\[
\G(R'/R)=\{\bu^7,\bu^8,\bu^{13},\bu^{14},\bu^{19}\}.
\] 
Radicals of $R$ are $\sqrt[2]{R}=R[\![\bu^8,\bu^{13},\bu^{14},\bu^{19}]\!]=R[\![\bu^8]\!]$ and 
$\sqrt[n]{R}=R'$ for $n>2$. Hence $\FG(R'/R)=\{\bu^8,\bu^{13},\bu^{14},\bu^{19}\}$. 
\end{example}

As observed in the classical case \cite[Lemma~10]{Ojeda-Rosales-2020}, equi-gcd numerical semigroup algebras with a single fundamental gap monomial
have interesting properties.

\begin{lem}\label{sing}
Let $R'/R$ be an equi-gcd numerical semigroup algebra in the variable $\bu$. A monomial 
$\bu^s\in\G(R'/R)$ is the only fundamental gap monomial if and only if 
any gap monomial is a radical of $\bu^s$.
\end{lem}
\begin{proof}
Assume first that $\FG(R'/R)=\{\bu^s\}$. Given $\bu^t\in\G(R'/R)\setminus\{\bu^s\}$, we have 
$\bu^t\not\in\sqrt[n_1]{R}$ for some $n_1\geq 2$. Hence $\bu^{tn_1}\in\G(R'/R)$. If $\bu^{tn_1}\neq\bu^s$,
there exists $n_2\geq 2$ such that $\bu^{tn_1n_2}\in\G(R'/R)$. Repeating this procedure, eventually we 
find $n\geq 2$ such that $\bu^{tn}=\bu^s$.

On the other hand, we assume that any gap monomial is a radical of $\bu^s$.  Given
$\bu^t\in\G(R'/R)\setminus\{\bu^s\}$, there exists $n\geq 2$ such that $\bu^{tn}=\bu^s$. Hence 
$\bu^t\not\in\sqrt[n]{R}$ and $\bu^t\not\in\FG(R'/R)$.
\end{proof}

\begin{prop}\label{fia}
Any numerical semigroup ring has finitely many equi-gcd coefficient rings, 
over which there is a single fundamental gap monomial.  
\end{prop}
	
\begin{proof}
We may work on a numerical semigroup ring $R'$ in the variable $\bu$ such that $\log_\bu R'\subset\N$ 
and $\gcd(\log_\bu R')=1$. Choose an integer $n>4$ such that $\bu^s\in R'$ for any $s\geq n$.
If the proposition was not true, there would exist an equi-gcd coefficient ring $R$, over which there is a single fundamental gap monomial $\bu^w$ for some $w>2n$. Let
	\[
	m=\begin{cases}
	\frac{w+1}{2}, & \text{ if } w \text{ is odd;}\\ \frac{w+2}{2}, & \text{ if } w \text{ is even}.
	\end{cases}
	\]
Then $\bu^m, \bu^{w-m}\in R'$ are not radicals of $\bu^w$. Hence $\bu^m, \bu^{w-m}\in R$ by Lemma~\ref{sing}.
However, $\bu^m\bu^{w-m}\in R$ contradicts to $\bu^w\in\G(R'/R)$.
\end{proof}

\begin{example}
It is shown in the logarithmic form \cite[Lemma~11]{Ojeda-Rosales-2020} that the numerical semigroup rings
$\kappa[\![\bu[\!]$, $\kappa[\![\bu^2,\bu^3[\!]$, $\kappa[\![\bu^3,\bu^4,\bu^5]\!]$, $\kappa[\![\bu^2,\bu^5]\!]$, 
$\kappa[\![\bu^3,\bu^5,\bu^7]\!]$ and $\kappa[\![\bu^4,\bu^5,\bu^7]\!]$ are all equi-gcd coefficient rings of 
$\kappa[\![\bu]\!]$, over which there is a single fundamental gap monomial. 
\end{example}

\begin{example}
Let $R$ be a non-trivial equi-gcd coefficient ring of $\kappa[\![\bu^2,\bu^3]\!]$. Then  $\FG(\kappa[\![\bu]\!]/R)=\FG(\kappa[\![\bu^2,\bu^3]\!]/R)$.   Non-trivial coefficient rings of $\kappa[\![\bu^2,\bu^3]\!]$, over which, there is a  single fundamental gap monomials are $\kappa[\![\bu^3,\bu^4,\bu^5]\!]$, $\kappa[\![\bu^2,\bu^5]\!]$, 
$\kappa[\![\bu^3,\bu^5,\bu^7]\!]$ and $\kappa[\![\bu^4,\bu^5,\bu^7]\!]$. 
\end{example}

\begin{cor}
Any numerical semigroup ring has finitely many equi-gcd coefficient rings, whose extensions 
are all intersections of radicals. 
\end{cor}
\begin{proof}
Let $R'$ be a numerical semigroup ring in the variable $\bu$. Given an equi-gcd coefficient ring 
$R$ whose extensions are all intersections of radicals, it suffices to prove that $\FG(R'/R)$ is a singleton.
Assume the contrary that $\FG(R'/R)$ contains monomials $\bu^t$ and $\bu^s$ such that $t<s$. Then the extension $R[\![\bu^s]\!]$ is not  an intersection of radicals, because it does not contain $\bu^t$.
\end{proof}

\begin{prop}\label{FGsing}
Let $R'/R$ be an equi-gcd numerical semigroup algebra in the variable $\bu$. Then $\FG(R'/R)$ is a singleton if
and only if all extensions of $R$ in $R'$ are  intersections of radicals. 
\end{prop}
\begin{proof}
Assume first that $\bu^s,\bu^t\in\FG(R'/R)$ and $t<s$. Then the extension $R[\![\bu^s]\!]$ is not 
 an intersection of radicals, because it does not contain $\bu^t$. On the other hand, we assume that $\FG(R'/R)=\{\bu^w\}$.
Consider an extension $R[\![\bu^{w_1},\ldots,\bu^{w_n}[\!]$, where $\bu^{w_1},\ldots,\bu^{w_n}\in\G(R'/R)$.
For each $\bu^{w'}\in\G(R'/R)\setminus R[\![\bu^{w_1},\ldots,\bu^{w_n}[\!]$, there exits $d'\in\N$ such that
$d'w'=w$. Hence $\bu^{w'}\not\in\sqrt[d']{R}$. We claim that $\bu^{w_i}\in\sqrt[d']{R}$ for all $i$. It then
follows that $R[\![\bu^{w_1},\ldots,\bu^{w_n}[\!]$ is the intersection of all such radicals of $R$. If 
$\bu^{w_i}\not\in\sqrt[d']{R}$, there exists $d_i\in\N$ such that $d_id'w_i=w$. Together with $d'w'=w$,
we obtain $d_iw_i=w'$ contradicting to $\bu^{w'}\not\in R[\![\bu^{w_1},\ldots,\bu^{w_n}[\!]$.

\end{proof}

\section*{Acknowledgments}
This work has been initiated during the visit of Raheleh Jafari to the Institute of Mathematics, Academia Sinica in 2018. The  authors would like to thank the institute for the great hospitality and support.


\end{document}